\documentclass[oneside,reqno,12pt]{amsart}
\usepackage[left=1in,top=1in,right=1in,bottom=1in]{geometry}
\usepackage{amsmath,amsfonts,amssymb,amsthm}
\usepackage{verbatim}
\usepackage{latexsym}
\newtheorem{theorem}{Theorem}[section]
\newtheorem{lemma}[theorem]{Lemma}

\newtheorem{definition}[theorem]{Definition}
\newtheorem{corollary}[theorem]{Corollary}

\newtheorem{claim}[theorem]{Claim}

\newcommand\RR{{\Bbb R}}
\newcommand\CC{{\Bbb C}}

\newcommand\NN{{\Bbb N}}
\newcommand\ZZ{{\Bbb Z}}

 \newcommand\HH{{\Bbb H}}

\newcommand{\note}[2][\null]{%
  \marginpar{\renewcommand{\baselinestretch}{1}\vspace{-1em}\hrule\vspace{3pt}%
  \tiny\raggedright{#2\ifx#1\null\else\\\hfill---
  {\em #1}\fi}\vspace{1.5em}}%
}

\begin{document}
\title{ 
Paley-Wiener  description  of $K$-spherical Besov spaces  on the Heisenberg group}

 \author{Azita Mayeli}
\address{Department of Mathematics and Computer Sciences,  Queensborough College, at City University of New York (CUN),
 NY 11364}
\email{amayeli@qcc.cuny.edu}

 \thanks{ }
\date{\today}

   \begin{abstract}  
   We characterize the Besov spaces associated to the Gelfand pairs on the Heisenberg group. The characterization is given in terms of bandlimited wavelet coefficients where the bandlimitedness is introduced using spherical Fourier transform. To obtain these results we develop an approach to the characterization  of Besov spaces in abstract Hilbert spaces  through  compactly  supported admissible functions. 
      \end{abstract}
  
  \maketitle

 \section{Introduction}\label{introduction}
 
  One of the main themes in Analysis is correlation between frequency content of a function and its smoothness. On the classical level the frequency is understood in terms of Fourier transform (or Fourier series) and smoothness is described in terms of Sobolev, Lipshitz, and Besov norms. 
  
  For these notions it was well understood (see \cite{Akh}, \cite{BL}, \cite{N}, \cite{Tr3}) that there exists a perfect balance between rate of approximation by bandlimited functions (by trigonometric polynomials) and smoothness described by Besov norms.
  
A form of a  harmonic analysis  which holds this balance true in general Hilbert spaces and manifolds was recently developed in \cite{Pes7}-\cite{Pes11} and \cite{gm4}.
In this papers Sobolev and Besov spaces on manifolds were associated with elliptic Laplace-Beltrami operators on manifolds. 

   We note that harmonic analysis associated with Besov norms, approximations by bandlimited vectors and   $K$-functions was considered in   \cite{Pes1}-\cite{Pes6}, \cite{KPes}.

   New developments in this direction were recently published in \cite{C1}-\cite{FM1},  \cite{gm4}, \cite{Pes7}. In \cite{FM1} 
  the authors characterize inhomogeneous Besov spaces on stratified Lie groups using some functional calculus and the spectral theory  for the sub-Laplacian operators on these groups.

 We base our work on the observation (see Definition \ref{Paley-Wiener functions}  below) that every time one has some kind of Fourier transform, the functions that are compactly supported on the ``frequencyÓ  side are natural generalizations of the classical bandlimited (Paley-Wiener) functions.
In the introduction we formulate main results obtained in our paper. The exact definitions of all notions are given in the text

   In the  introduction we formulate  main results obtained
in our paper. The exact definitions of all notions are given in
the text. 

We start with a   self-adjoint positive definite operator $\Delta$   in  a Hilbert space $\mathcal{H}$ and consider its positive root
$D=\Delta^{1/2}$.  
The domain $\mathcal{D}_{s}, s\in \mathbb{R},$ of the operator $D^{s}, s\in \mathbb{R},$ plays the role of the
Sobolev space.   In what follows, the notation $\|\cdot\|$  means $\|\cdot\|_{\mathcal H}$. We define the following graph norm for the Sobolev spaces $\mathcal{D}_{r},$ as domain of the operator $\Delta^{r/2}$:
\begin{align}\notag
\|f\|_r= \|f\| + \|\Delta^{r/2} f\|~.
\end{align}

The inhomogeneous  Besov space $B_{2,q}^\alpha= B_{2,q}^\alpha(\Delta)$ was introduced as an interpolation space between the Hilbert space $\mathcal H$ and Sobolev space $\mathcal D_{r/2}$ where  
 $r$ can be any natural number such that $0<\alpha<r, 1\leq
q<\infty$, or  $0\leq\alpha\leq r,q= \infty$ (\cite{BL,BB,KPes,KPS,Tr3}). It is known that 
the Besov space can be characterized as space of all functions in $\mathcal H$ whose Besov norm can be described in terms of modulus of continuity in terms of wave semigroup
 $e^{itD}$ (\cite{BL,BB,KPS,Pes7,Pes11,Tr3}).

    The notion of the   Paley-Wiener  spaces for the abstract Hilbert space  $\mathcal H$ associated to the positive self-adjoint operator $D$ is given as below.    
According
to  the  spectral theory \cite{BS},  there exists a unitary operator $\mathcal{F}_{D}$ from $\mathcal{H}$ onto  a Hilbert space $X$, where $X=\int X(\lambda )dm (\lambda )$ is a direct integral of 
Hilbert spaces $X(\lambda )$ and is the space of all $m $-measurable
  functions $\lambda \rightarrow x(\lambda )\in X(\lambda ) $, for which the norm

$$\|x\|_{X}=\left ( \int ^{\infty }_{0}\|x(\lambda )\|^{2}_{X(\lambda )}
 dm (\lambda ) \right)^{1/2} $$ is finite.
The unitary
operator $\mathcal{F}_{D}$  
transforms domain of $D^{k}, k\in \mathbb{N},$ onto $X_{k}=\{x\in
X|\lambda ^{k}x\in X \}$ with norm

\begin{equation}
\|x(\lambda)\|_{X_{k}}= \left (\int^{\infty}_{0}
 \lambda^{2k}\|x(\lambda )\|^{2}_{X(\lambda )} dm
 (\lambda ) \right )^{1/2}~ . 
 \end{equation}
Besides $\mathcal{F}_{D}(D^{k} f)=
 \lambda ^{k} (\mathcal{F}_{D}f), $ if $f$ belongs to the domain of
 $D^{k}$.   
 
 The following definition can be found in \cite{Pes3} and \cite{PesTransAMS}. 

\begin{definition}\label{Paley-Wiener functions}  Let
  $D$ be the same as above.  Then  we  say  a vector $f$
from $\mathcal{H}$ belongs to the Paley-Wiener space $PW_\omega=PW_{\omega
}(D)$ if the support of the spectral Fourier transform
$\mathcal{F}_{D}f$ belongs to $[0, \omega]$. For a vector $f\in
PW_{\omega }$ the notation $\omega_{f}$ will be used for a
positive number such that $[0, \omega_{f}]$ is the smallest
interval which contains the support of the spectral Fourier
transform $\mathcal{F}_{D}f$. We call the vectors in Paley-Wiener spaces  {\it bandlimited} functions.
\end{definition}

The goal of this article  is to realize the above notions for the Heisenberg group and  describe the Besov norms for the  group in terms of {\it bandlimited } and {\it admissible (wavelet)} functions .

   Let $\HH_n \cong \RR^n\times \RR^n\times \RR$ denote the $n$-dimensional Heisenberg group and $K$ is a Lie compact subgroup of $U(n)$, the group of unitary automorphisms on $\HH_n$.  We let the space  $L_K^2(\HH_n)$ denote  the space of all functions $f$ in $L^2(\HH_n)$ which are $K$-invariant:  $f(kw)=f(w), ~ \forall k\in K, \ \forall w\in \HH_n$.
   
   For functions in $L^2_K(\HH_n)$ one can introduce the $K$-spherical  transform $\mathcal F$. It is known that $\mathcal F$ is a unitary operator from $L^2_K(\HH_n)$
   onto a certain $L^2$ space of functions defined on $\RR^*\times \NN^n$, $\RR^*=\RR/\{0\}$. We say a function $f\in L^2_K(\HH_n)$ is a Paley-Wiener function if $\mathcal F(f)$ has compact support in $\RR^*\times \NN_n$. 
   
   In what follows the $l_w$ and $\delta_a$ are standard translation and dilation operators on $\HH_n$.    We call $\Psi\in L^2_K(\HH_n)$ a wavelet if the measurable coefficient map $W_{f,\Psi}: (w,a)\mapsto \langle  f, l_w\delta_a \Psi\rangle$ is  square integrable.

   Our main result is Corollary \ref{MainResult} as following:

   {\bf Main result.} {\it
Let $(K, \HH_n)$ be a Heisenberg  Gelfand pair, i.e., $L_K^1(\HH_n)$ is a commutative algebra with convolution operator. Then there exists a bandlimited wavelet  
 $\Psi\in L^2_K(\HH_n)$ such that for any $f\in L^2_K(\HH_n)$ the following holds true.  For any  $f\in L_K^2(\HH_n)$

   \begin{align}\notag
\|f \|_{B_{2,q}^\alpha}  
\asymp  
   \| f\|+ \left(\sum_{j\geq0} \left( 2^{-j((n+1)-\alpha/q)}  \|\mathcal F(f) A_{2^{2j}}\mathcal F(\Psi)\|\right)^q
\right)^{1/q}
\end{align}    
where $\delta_{2^{-j}}  \Psi(w) = 2^{j(n+1)}\Psi(2^jw)$ for all $w\in \HH_n$, and for $a>0$, $A_a$ is a unitary dilation. 
 }\\
 
 The above equivalency is understood in this sense that $f$ is a  $K$-spherical Besov function  if and only if the sum is finite and  two norms are equivalent. 
We will use     the following main technical lemmas for the proof of our main results.

 \begin{lemma}\label{mainLemma1} Let $\mathcal H$ be a general Hilbert space. 
  Then for any $f\in \mathcal H$  there exists a sequence of bandlimited functions
$f_{j}:=f_{j}(f)\in PW_{2^{j+1}}$  and $g\in PW_1$,  and  
 a sequence of operators $S_j: \mathcal H \rightarrow PW_{2^{j+1}}$ and  $S:  \mathcal H \rightarrow PW_1$ such that in $\mathcal H$ 
      $$f=S(g)+\sum_{j=0}^\infty S_j(f_{j}) .$$
\end{lemma}

 \begin{lemma}\label{mainLemma2} Let $\mathcal H$ be a general Hilbert space and $f\in \mathcal H$. Then   there exist a sequence of bandlimited functions
$f_{j}:=f_{j}(f)\in PW_{2^{j+1}}$  and $g\in PW_1$ such that
  for $\alpha>0, 1\leq
q< \infty$, the Besov norm is equivalent to 
\begin{equation}
  \|g\|+ \left(\sum_{j=0}^{\infty}\left(2^{j\alpha
}\|f_{j}\|\right)^{q}\right)^{1/q}~~ \text{and,} ~~\quad
  \|f\|+ \left(\sum_{j=0}^{\infty}\left(2^{j\alpha
}\|f_{j}\|\right)^{q}\right)^{1/q} ~,~ \label{normequiv}
\end{equation}

   and the equivalency of norms also holds for $q=\infty$ with the standard convention. 
 \end{lemma}

 The existence of bandlimited atoms $f_j(f)$ was already proved on manifolds and general Hilbert spaces in \cite{Pes8,Pes11}. However, in the present paper we develop this result further. Namely, we give a ``constructive" description of such atoms $f_j(f)$ which are  ``infinitely smooth" on the space and enjoy ``all vanishing moments" property. This implies, as usual, that our atoms have perfect localization on the frequency side and a reasonable localization on the space (Heisenberg group). \\

   The outline of this paper is as following. After introducing some preliminaries and notations we prove the main Lemmas \ref{mainLemma1} and \ref{mainLemma2} in Section \ref{ProofOfLemmas} and hence we establish  an equivalent Besov norm on general Hilbert spaces in terms of  ``admissible" functions. In Section \ref{The Heisenberg group} we introduce  the Gelfand pairs and $K$-spherical (Gelfand) transform on the Heisenberg group associated to the sub-Laplacian on this group. We conclude this article with the proofs of our main results in Theorem \ref{Besov-interms-of-wavelet-coefficients} and Corollary \ref{MainResult}.

\section{Preliminaries and Notations}\label{notations}

 We introduce Besov spaces on $\mathcal H$,   
$\mathbf{B}^{\alpha}_{2,q}=\mathbf{B}^{\alpha}_{2,q}(\Delta), \alpha>0, 1\leq  q\leq \infty,$
  using Peetre's interpolation $K$-functions (\cite{BL,BB,KPes,KPS,Tr3}).
  That is 
  
\begin{equation}\notag
\mathbf{B}^{\alpha}_{2,q}=\left(\mathcal{H},\mathcal{D}_{r/2}\right)^{K}_{\alpha/r,q},\label{Besov
norm}
\end{equation}
where $r$ can be any natural number such that $0<\alpha<r$ for $ 1\leq
q<\infty$, or  $0\leq\alpha\leq r$ for $q= \infty$. 

  For any $f\in \mathcal H$, let $K(\cdot, f, \mathcal H, \mathcal{D}_{r/2})$ be the Peetre's $K$- functional for the pair $(\mathcal H, \mathcal{D}_{r/2}$)  given on $\RR^+$ by 
 \begin{align}
 K\left(t, f,\mathcal H,  \mathcal{D}_{r/2} \right)=\inf_{g\in
 \mathcal{D}_{r/2} }\left(\|f-g\|+t\|g\|_{r}\right), t>0, ~~ f\in \mathcal H~.
 \end{align}

We introduce the following functionals
$$
\Phi_{\theta,q}^{\varepsilon}(\varphi(t))=\left(\int_{0}^{\varepsilon}\left(t^{-\theta}
\varphi(t)\right)^{q}\frac{dt}{t}\right)^{1/q}, 1<\theta<\infty,
1\leq q<\infty,
$$
and
$$
\Phi_{\theta,\infty}^{\varepsilon}(\varphi(t))=\sup_{0\leq t\leq
\varepsilon}t^{-\theta}(\varphi(t)), ~~q=\infty.
$$

The  spaces $(\mathcal H,\mathcal{D}_{r/2})^{K}_{\theta, q}, 0<\theta<1,
1\leq q\leq \infty,$ are the sets of  abstract functions in $\mathcal H$ for which  the
following norm
\begin{equation}\notag
\|f\|+\Phi_{\theta,q}^{\varepsilon}\left(K\left(t, f, \mathcal H,
\mathcal{D}_{r/2}\right)\right) \quad f\in \mathcal H
\end{equation}
is finite. The fact that
$H^{r}\subset \mathcal H$ implies that for any $\varepsilon
>0$ the following two norms are equivalent. 

$$ \|f\|+\Phi_{\theta,q}^{\varepsilon}\left(K\left(t, f, \mathcal H,
\mathcal{D}_{r/2}\right)\right) \asymp
\Phi_{\theta,q}^{\infty}\left(K\left(t, f, \mathcal H,
\mathcal{D}_{r/2}\right)\right) \quad f\in \mathcal H~.
$$
\\

 It is known result  in \cite{BL,BB,Pes7} that this  Besov norm 
 can be described in terms of a modulus of
continuity constructed in terms  of the wave semigroup  $e^{itD}$.  
 
\begin{theorem}\label{equivalency-w-modulus of continuity}
 Let $ \alpha < r \in \NN$. The
     norm of the Besov space $B_{2,q}$ on $\mathcal H$ is equivalent to

  \begin{align}\label{fraction-norm}
 \| f\| + \left(\int_0^1  \left(s^{-\alpha } \Omega_{r}(s,f)\right)^q ds/s \right)^{1/q}
\end{align} 
for $1\leq q<\infty$  and equivalent to 
   \begin{align}\notag
 \| f\| +  \sup_{0<s<1}  \left(s^{-\alpha } \Omega_{r}(s,f)\right) ds/s 
\end{align} 
for $q=\infty$, where modulus of continuity is introduced as 
\begin{equation}\notag
\Omega_{r}(s,\>f )=\sup_{0<\tau\leq s} \| \left( I-e^{i\tau\sqrt{\Delta}}  \right)^r f\|~.\end{equation}
\end{theorem}

\subsection{Functional Calculus} 
 By the spectral theory, suppose that $\Delta$ has the unique  spectral resolution or decomposition  $P$ of the identity   

\begin{align}\label{spectral decomposition}
\Delta= \int_0^\infty \xi dP_{\xi} ~.
\end{align}
 $dP$ is a projection-valued measure concentrated on the spectrum of $\Delta$, $\sigma(\Delta)=(0,\infty)$,  with orthogonal projections $P_\xi$ on $ \mathcal H$ with $P_{\{0\}}( \mathcal H)= 0$ and $P_{\sigma(\Delta)}=I$. Therefore by the spectral theory, 
  for any  $f\in \mathcal D(\Delta)$ and $g\in  \mathcal H$
 
 \begin{align}\notag
 \langle \Delta f, g\rangle = \int_0^\infty \xi~  d(P_\xi f, g) ,
 \end{align}
and 
   
  $$\mathcal D(\Delta)=\left\{ f\in  \mathcal H : ~~~ \|\Delta f\|^2:=\int_0^\infty \xi^2 d(P_\xi f, f)<\infty\right\}.$$

   For $\beta$  a bounded Borel function on $\sigma(\Delta)$, we define the commutative integral  operator  $\beta(\Delta)$ by

\begin{align}\notag
\beta(\Delta):=\int_0^\infty \beta(\xi)dP_{\xi};
\end{align} 
by the spectral theory 
this is a bounded  operator  with domain 
  $$\mathcal D(\beta(\Delta))=\left\{ f\in  \mathcal H: ~~~\|\beta(\Delta)f\|^2:=  \int_0^\infty |\beta(\xi)|^2 d(P_\xi f, f)<\infty\right\}.$$

The operator norm is $\|\beta(\Delta)\|=\|\beta\|_\infty$ and the following hold:\\
 \noindent
(a)  $\beta\gamma(\Delta)=\beta(\Delta)\gamma(\Delta)=\gamma(\Delta)\beta(\Delta)$\\
(b) $\beta(\Delta)^*=\bar\beta(\Delta)$\\
(c) for any $f\in\mathcal D(\beta(\Delta))$ and $g\in \mathcal H$
\begin{align}\notag
 \langle \beta(\Delta)f, g\rangle = \int_0^\infty \beta(\xi)~  d(P_\xi f, g)~.
 \end{align}
  
 $\beta$ is real-valued, then the operator $\beta(\Delta)$ is self-adjoint by (b). The operator is positive definite if $\beta$ takes its values in $\RR^+$.   \\
 
 Throughout this paper, by $A\asymp B$ we shall mean that  $A, B$  are positive numbers and that there are positive constants $c_1, c_2$ such that $c_1 A\leq B \leq c_2 A$. Similarly, we say $A\preceq B$ if there exists $c>0$ such that $A\leq cB$.  \\
 
In a complete analogy to the Fourier transform of a function on $\RR$,  we shall use   $\hat\psi$ for the functions given  on the spectrum of $\Delta$.

 \section{Proof of   Lemma  \ref{mainLemma1} and Lemma \ref{mainLemma2}}\label{proof of main theorem}\label{ProofOfLemmas}

   \begin{proof}[Proof of Lemma \ref{mainLemma1}]
   Suppose $\hat\phi, \hat\psi \in L^\infty(0,\infty)$  with $\text{supp}(\hat\phi)\subseteq [0,1]$ and 
   $\text{supp}(\hat\psi)\subseteq [1/2,2]$ for which the following resolution of the identity holds: 
   
\begin{align}\label{resolution of identity}
| \hat\phi(\xi)|^2+\sum_{j\in \ZZ^+} | \hat\psi_j(\xi)|^2= 1 \quad \forall ~ \xi\in \RR^+ 
\end{align}  
where $\hat\psi_j(\xi):= \hat\psi(2^{-2j}\xi)$. Thus  
  $\text{supp}(\hat\psi_j) \subseteq [2^{j-1},  2^{j+1}]$.
 
 Applying the spectral theory for (\ref{resolution of identity}), the following version of Calder\'on decomposition, in complete analogy to the Euclidean setting, holds:   
 
 \begin{align}\label{resolution of identity operator}
 \hat\phi(\Delta)^* \hat\phi(\Delta)f+\sum_{j\in \ZZ^+}  \hat\psi_j(\Delta)^*\hat\psi_j(\Delta)f= f ~\quad \forall f\in \mathcal H
\end{align}
 where the series converges in $\mathcal H$.  Now define $S_j:= \hat\psi_j(\Delta)$ and $f_j:=S_j(f)$. Then the functions $f_j$ are in $PW_{2^{j+1}}$ and the assertions of the lemma hold. 
\end{proof}

 For the proof of Lemma \ref{mainLemma2} we need  the following  two  technical lemmas.

\begin{lemma}\label{technical lemma2}
   For any $\tau>0$ and natural number $r$,   in the operator norm 
  \begin{align}\notag 
  \| \left( I-e^{i\tau\sqrt{\Delta}}  \right)^r\hat\phi(\Delta)\| &\preceq \tau^r,\quad \text{and}\\\notag
    \| \left( I-e^{i\tau\sqrt{\Delta}}  \right)^r\hat\psi_j(\Delta)\| &\preceq  \tau ^r2^{(j+1)r/2} \quad j\in \ZZ^+ .
  \end{align}
   \end{lemma}

\begin{proof} We  prove the   inequality for $\hat\psi_j$ and the proof  for $\hat\phi$  follows in an analogy way. 
Since  $\text{supp}  \hat\psi_j \subseteq[ 2^{j-1} , 2^{j+1}]$,

\begin{align}\label{second line}
  \| \left( I-e^{i\tau\sqrt{\Delta}}  \right)^r\hat\psi_j(\Delta)\|   &=  \sup_{ 2^{j-1}\leq \xi \leq 2^{j+1}} | \left( I-e^{i\tau\sqrt{\xi}}  \right)^r  \hat\psi_j(\xi)|
  \\\label{boundedness of psi}
&\preceq \sup_{ 2^{j-1}\leq \xi \leq 2^{j+1}}| (1-e^{i\tau\sqrt{\xi}})^r|\\ \label{FTC}
  &\preceq   ~\tau^r2^{(j+1)r/2}~. 
\end{align}
 
Note that to pass  from (\ref{second line}) to (\ref{boundedness of psi}) we use  $|\hat\psi_j(\xi)|\leq 1$ which is driven from  (\ref{resolution of identity}).   
And,  to pass from  (\ref{boundedness of psi})  to (\ref{FTC}) 
we use  the Fundamental Theorem of Calculus for the function $h(x)= e^{ix\sqrt{\xi}}$ on   $[0,\tau]$.
\end{proof}

\begin{lemma}\label{second technical lemma}
Let $r\in\NN$,  $m, k\in \RR$  such that  $k+m\geq 0$ and $k<0$.  
Let $f\in \mathcal H$. Define 
 $w_j= 2^{kj}$ and $c_j=2^{mj}$ for $j\in \ZZ$. Then for any $1\leq \tilde q <  \infty$ 
\begin{align}\notag
\Omega_r(s,f)^{\tilde q}  \preceq s^{\tilde q r} \sum_{j=-1}^\infty \left(2^{j\tilde q r/2} w_j c_j^{\tilde q}
\| \hat\psi_j(\Delta)f\|^{\tilde q}\right)\quad \quad \forall ~ s\in (0,1] .
\end{align}

 \end{lemma}
 \begin{proof} Take $\hat\psi_{-1}:=\hat\phi$ and let  $f\in B_{2,q}^\alpha$.  By Applying the decomposition (\ref{resolution of identity operator})  to  $f$,  for any $\tau\leq s$ 
 we have 
 \begin{align}\label{x}
   \|\left( I-e^{i\tau\sqrt{\Delta}}  \right)^r f\|&=\|\sum_{j=-1} \left( I-e^{i\tau\sqrt{\Delta}}  \right)^r\hat\psi_j(\Delta)^*\hat\psi_j(\Delta)f\|
 & \leq  \sum_{j=-1} \|\left( I-e^{i\tau\sqrt{\Delta}}  \right)^r\hat\psi_j(\Delta)^*\hat\psi_j(\Delta)f\|~. 
\end{align} 
 
  Taking supremum    over  $\tau$  in (\ref{x})  yields 
 \begin{align}\notag
 \sup_{\tau\leq s} \|\left( I-e^{i\tau\sqrt{\Delta}}  \right)^r f\|
&\leq \sup_{\tau\leq s} \sum_{j=-1}  \|\left( I-e^{i\tau\sqrt{\Delta}}  \right)^r\hat\psi_j(\Delta)^*\hat\psi_j(\Delta)f\|\\\notag
& \leq \sum_{j=-1} \sup_{\tau\leq s} \|\left( I-e^{i\tau\sqrt{\Delta}}  \right)^r\hat\psi_j(\Delta)^*\hat\psi_j(\Delta)f\|    
 \\\notag 
 & \preceq  \sum_{j=-1} w_j ~c_j  \sup_{\tau\leq s}\|\left( I-e^{i\tau\sqrt{\Delta}}  \right)^r\hat\psi_j(\Delta)^*\hat\psi_j(\Delta)f\| ~ \quad (w_jc_j\geq 1)~.
\end{align} 
 Therefore  for $\tilde q>1$
  \begin{align}\label{w-c-supp}
\left(\sup_{\tau\leq s} \|\left( I-e^{i\tau\sqrt{\Delta}}  \right)^r f\|\right)^{\tilde q}
\leq \left(  \sum_{j=-1} w_j ~c_j  \sup_{\tau\leq s}\|\left( I-e^{i\tau\sqrt{\Delta}}  \right)^r\hat\psi_j(\Delta)^*\hat\psi_j(\Delta)f\|\right)^{\tilde q}.
\end{align}

  The H\"older inequality for  the series  in  (\ref{w-c-supp}) and for  the pair $(\tilde q, \cfrac{\tilde q}{\tilde q-1})$    with the  weights
$w_j$ 
yields  the following estimations   up to some constants independent of $f$:
\begin{align}\notag
(\ref{w-c-supp}) &\preceq  \left(\sum_{j=-1} w_j\right)^{{\tilde q}-1} \sum_{j=-1} w_j c_j^{\tilde q}  \sup_{\tau\leq s}\|\left( I-e^{i\tau\sqrt{\Delta}}  \right)^r\hat\psi_j(\Delta)^*\hat\psi_j(\Delta)f\|^{\tilde q}\\\notag
&\preceq    \sum_{j=-1} w_j c_j^{\tilde q}  \sup_{\tau\leq s}\|\left( I-e^{i\tau\sqrt{\Delta}}  \right)^r\hat\psi_j(\Delta)^*\hat\psi_j(\Delta)f\|^{\tilde q}\\\label{technical-lemma}
&\preceq    \sum_{j=-1} w_j c_j^{\tilde q}~ \|\hat\psi_j(\Delta)f\|^{\tilde q}~ \sup_{\tau\leq s}\|\left( I-e^{i\tau\sqrt{\Delta}}  \right)^r \hat\psi_j(\Delta)^*\|^{\tilde q}\\\label{before-last-line}
&\preceq   \sum_{j=-1} w_j c_j^{\tilde q}~ \|\hat\psi_j(\Delta)f\|^{\tilde q}~  \sup_{\tau\leq s}(\tau^r 2^{(j+1)r/2})^{\tilde q}\\\notag
&\preceq   s^{{\tilde q}r} \sum_{j=-1} 2^{jr{\tilde q}/2} w_j c_j^{\tilde q}~ \|\hat\psi_j(\Delta)f\|^{\tilde q} ~.
\end{align} 
 
We note that to  pass from  (\ref{technical-lemma}) to (\ref{before-last-line}) we applied  Lemma \ref{technical lemma2} for $\overline{\hat\psi}$.  Interfering the preceding estimations in  
 (\ref{w-c-supp})  we achieve  the result. 
The assertion for $\tilde q=1$  is  obtained with a similar  argument.  
 \end{proof}
 
\begin{proof}[Proof of Lemma \ref{mainLemma2}.]
   Let $f_j$ be as above. 
 We prove that for these functions the equivalency (\ref{normequiv})  of the Lemma \ref{mainLemma2} hold. We shall prove this in two parts. \\

{\bf Part I.} For any $f\in \mathcal H$, if $\{2^{j\alpha} f_j\}_j\in l^q(\ZZ^+, \mathcal H)$, then $f\in B_{2,q}^\alpha$ and 

 \begin{align}\label{fraction-norm}
 \| f\| + \left(\int_0^1  \left(s^{-\alpha }  \Omega_r(f,s)\right)^q ds/s \right)^{1/q}
\preceq 
  \| f\| +\left( \sum_{j\in\ZZ^+} ( 2^{j\alpha} \| f_j\|)^q \right)^{1/q}~.
\end{align}
 And, for $q=\infty$ the result holds. 
 
 {\bf Proof of Part I.} Let $1\leq q<\infty$ and take $\tilde q=q$ 
 in Lemma \ref{second technical lemma}. Let   $r\leq 2\alpha$ and 
$k$ and $m$ satisfy the inequality $k+mq\leq q(\alpha-r/2)$. (In fact, there exists  a large class   of  pairs  $(k,m)$ such that $k+m\geq 0$ and simultaneously  satisfy   the inequality.)
 By Lemma \ref{second technical lemma}, 
 
  \begin{align}\label{q-power-estimation}
      \left( \sup_{\tau\leq s} \|\left( I-e^{i\tau\sqrt{\Delta}}  \right)^r f\|\right)^q  \preceq
      s^{qr} \sum_{j=-1} 2^{jrq/2} w_j c_j^q~   \|\hat\psi_j(\Delta)f\|^q  ~.
  \end{align} 
By integrating  the both sides of  (\ref{q-power-estimation})    on $[0,1]$   with respect to the   measure $s^{-\alpha q}\frac{ds}{s}$  
we get
 \begin{align}\notag
  \int_0^1  s^{-\alpha q} \left( \sup_{\tau\leq s} \|\left( I-e^{i\tau\sqrt{\Delta}}  \right)^r f\|\right)^q ds/s 
 &\preceq 
   \left(\int_0^1  s^{(r-\alpha) q}   ds/s\right)
 \sum_{j=-1} 2^{jqr/2} w_j c_j^q~  \|\hat\psi_j(\Delta)f\|^q  \\\notag
&=  \cfrac{1}{q(r-\alpha)} \sum_{j=-1} 2^{jqr/2} w_j c_j^q~  \|\hat\psi_j(\Delta)f\|^q  ~.
\end{align}
Hence 
\begin{align}\label{sum}
  \left( \int_0^1 \left( s^{-\alpha} \sup_{\tau\leq s} \|\left( I-e^{i\tau\sqrt{\Delta}}  \right)^r f\|\right)^q ds/s\right)^{1/q}
  \preceq
 \left( \sum_{j=-1} 2^{jqr/2} w_j c_j^q~  \|\psi_j(\Delta)f\|^q  \right)^{1/q}.
 \end{align}
 
To complete the proof of part I,    define 
\begin{align}\notag
A_j:=
\begin{cases}
2^{jr/2} w_j^{1/q} c_j \|\psi_j(\Delta)f\|^q & \text{ if~~ $j=-1$}
\\\notag
0 & \text{ if ~~ $j\geq 0$} \  , \end{cases}
\end{align}
and  
 \begin{align}\notag
B_j:=
\begin{cases}
0 & \text{ if~~ $j=-1$}
\\\notag
 2^{jr/2} w_j^{1/q} c_j  \|\psi_j(\Delta)f\|^q & \text{ if ~~ $j\geq 0$} .\end{cases}
\end{align}

 We rewrite   the right hand side of (\ref{sum}) as follows. 

\begin{align}\notag
\left(\sum_{j=-1} 2^{jqr/2} w_j c_j^q~  \|\hat\psi_j(\Delta)f\|^q \right)^{1/q}&=  \left(\sum_{j=-1}   \mid A_j+B_j\mid^q\right)^{1/q}\\\label{47}
&\leq 
 \left(\sum_{j=-1}   \mid A_j\mid^q\right)^{1/q}+  \left(\sum_{j=-1}   \mid B_j\mid^q\right)^{1/q}
 \end{align}
 Substituting back $A_j$ and $B_j$ in above and  using $2^{-r/2} w_{-1}^{1/q} c_{-1}\leq 1$, and $2^{jqr/2} w_j c_j^q\leq 2^{j\alpha q}$ 
 we get 
   \begin{align}\notag
 (\ref{47})
 &\leq 2^{-r/2} w_{-1}^{1/q} c_{-1} \|\hat\psi_{-1}(\Delta)f\|  +\left( \sum_{j=0} 2^{jqr/2} w_j c_j^q \|\hat\psi_j(\Delta)f\|^q \right)^{1/q}\\\notag
 &=      ~ 2^{-r/2} w_{-1}^{1/q} c_{-1}\|\hat\phi(\Delta)f\| +\left( \sum_{j=0}   2^{jqr/2} w_j c_j^q
 \|\psi_j(\Delta)f\|^q \right)^{1/q}\\\notag
 &\leq ~ \|f\| +\left( \sum_{j=0}   2^{jqr/2} w_j c_j^q  \|\hat\psi_j(\Delta)f\|^q \right)^{1/q}\\\notag
&\leq  \|f\| +\left( \sum_{j=0}   2^{j\alpha q} \|\hat\psi_j(\Delta)f\|^q \right)^{1/q}=   \|f\|_{B_{2,q}^\alpha}~.
 \end{align}
 This completes the proof of $I$ for $1\leq q<\infty$. By 
 Lemma \ref{second technical lemma}, for $\tilde q=1$ we have 
 
 \begin{align}\label{sup-kern}
\Omega_r(s,f)\preceq  s^r \sum_{j=-1} 2^{jr/2} w_j ~c_j  \|\hat\psi_j(\Delta)f\|~.
\end{align} 
 By multiplying both sides of (\ref{sup-kern}) by $s^{-\alpha}$ and taking the  supremum    over $0<s<1$  we get     
  \begin{align}\notag
 \sup_{ 0< s<1} s^{-\alpha} \Omega_r(s,f) & 
 \preceq ( \sup_{ 0< s<1} s^{r-\alpha})   \sum_{j=-1} 2^{jr/2} w_j ~c_j  \|\hat\psi_j(\Delta)f\| \quad \quad (r-\alpha>0) \\\notag
  &\preceq 
    \sum_{j=-1} 2^{jr/2} w_j ~c_j  \|\hat\psi_j(\Delta)f\| \\\notag
      &\preceq   \left(\sup_{j\geq -1} 2^{j\alpha}\|\hat\psi_j(\Delta)f\|\right)\left(    \sum_{j=-1} 2^{-j(\alpha-r/2)} w_j ~c_j\right) 
 \end{align}
 Recall that here $\tilde q=q=1$ and $r-\alpha>0$.  With these restrictions and 
   $k+m\leq \alpha - r/2$  the sum  $  \sum_{j=-1} 2^{-j(\alpha-r/2)} w_j ~c_j$ is finite. 
   Therefore  
\begin{align}\notag
 \sup_{ 0< s<1} s^{-\alpha} \Omega_r(s,f)    \preceq 
 ~ \sup_{j\geq -1} 2^{j\alpha}\|\hat\psi_j(\Delta)f\| =  \sup_{j\geq 0} 2^{j\alpha}\|\hat\psi_j(\Delta)f\|  + \|f\|~.
 \end{align}
 This completes  the proof of Part I for $q=\infty$.

 {\bf Part II.} For any $f\in \mathcal H$
 
  \begin{align}\label{converse of first theorem} 
  \| f\| +\left( \sum_{j\in\ZZ^+} ( 2^{j\alpha} \| \hat\psi_j(\Delta)f\|)^q \right)^{1/q} \preceq  \| f\| + \left(\int_0^1  \left(s^{-\alpha }  \Omega_r(f,s)\right)^q ds/s \right)^{1/q}
\end{align}
 And, for $q=\infty$ the statement also holds true. 
 
  {\bf Proof of Part II.}
Put $c=\int_0^1 s^{-\alpha} |1-e^{is/2}|^{4r} ds $. Then $c >0$ and without lose of generality we  assume that 
$c=1$. Let $1/4\leq t\leq 1$.  By substituting $s\mapsto 2s \sqrt{t}$ in the above integral  we get 
  $$1 = \int_0^{\cfrac{1}{2\sqrt{t}}}  ~(2s\sqrt{t})^{-\alpha} |1-e^{is\sqrt{t}}|^{4r} 
 \sqrt{t} ds. $$
 
  Since $1/\sqrt{t}\leq 2$ and $\sqrt{t}\leq 1$, the followings hold for any $M>0$ up to some constants independent of $t$. 
 
   $$1=  \int_0^{\cfrac{1}{2\sqrt{t}}}  ~(2s\sqrt{t})^{-\alpha} |1-e^{is\sqrt{t}}|^{4r} 
 \sqrt{t} ds\preceq      \int_0^1  ~(s\sqrt{t})^{-\alpha} |1-e^{is\sqrt{t}}|^{4r} 
 ds\preceq  \int_0^1  ~s^{-\alpha} (\sqrt{t})^{-(\alpha+M)} |1-e^{is\sqrt{t}}|^{4r} 
 ds ~.$$  
Now define  $\xi(t):=   \int_0^1  ~s^{-\alpha} (\sqrt{t})^{-(\alpha+M)} |1-e^{is\sqrt{t}}|^{4r} 
 ds $ on  $1/4\leq t\leq 1$ and zero elsewhere.   The map $\xi$ is bounded and $\xi(t)\geq 1$.  By the spectral theory  for  $\Delta$ the operator $\xi(\Delta)$ is bounded on $\mathcal H$ and  we obtain the following      in the weak sense. 
    \begin{align}\label{operator inequality}
    I \leq   \int_0^{1}   s^{-\alpha}  (\sqrt{\Delta})^{-(\alpha+M)} \mid 1-e^{is\sqrt{\Delta}}\mid ^{4r} ds  
   \end{align}
 where $I$ is the identity operator on the Hilbert space $\mathcal H$. Therefore for any $g\in \mathcal H$

   \begin{align}\label{operator inequality-for g}
    \|g\|^2  \leq   \int_0^{1}   s^{-\alpha}  \|(\sqrt{\Delta})^{-(\alpha+M)/2}  (1-e^{is\sqrt{\Delta}})^{2r} g\|^2
   ~ ds ~.
   \end{align} 
   
   Take $g:= \hat\psi_j(\Delta)f$, $j\geq 0$. 
  An application of (\ref{operator inequality}) and  Lemma \ref{technical lemma2} gives
  \begin{align}\notag 
\| \hat\psi_j(\Delta)f \|^2 &\leq \int_0^{1}   s^{-\alpha} \| (\sqrt{\Delta})^{-(\alpha+M)/2}   (1-e^{is\sqrt{\Delta}})^{2r}   \hat\psi_j(\Delta)
 f\|^2~  ds \\\notag
 &\leq \int_0^{1}   s^{-\alpha} \|(1-e^{is\sqrt{\Delta}})^{r} 
 f\|^2
 \| (\sqrt{\Delta})^{-(\alpha+M)/2} (1-e^{is\sqrt{\Delta}})^{r} \hat\psi_j(\Delta)\|_{op}^2~  ds  \\\notag
 &  \preceq \int_0^{1}   s^{-\alpha} \Omega_r(s, f)^2~  \left( 2^{-j(\alpha+M)/2}   2^{jr} \right) s^{2r} ~ds
 \\\notag
&=  2^{-j(\alpha+M-2r)/2} ~  \int_0^{1}  (s^{-\alpha} \Omega_r(s, f))^2 s^{2r+\alpha}~ds\\\label{apply}
&= 2^{-j(\alpha+M-2r)/2} ~  \| s\mapsto s^{-\alpha} \Omega_r(s, f)\|_{L^2((0,1],dm(s))}^2
 \end{align}
 \noindent
 with $dm(s)= s^{1r+\alpha}$. Take
  $F(s):=s^{-\alpha} \Omega_r(s, f)$ and $G(s):=1$, $0<s\leq 1$. The inverse of the  H\"older inequality for $p=2$ implies that 
 
 \begin{align}\notag
 \|F\|_{L^2}\|G\|_{L^{-1}}\leq \|FG\|_{L^1}  ,  
 \end{align}
  equivalently, 
 
  \begin{align}\notag
 \|F\|_{L^2} \leq \|G\|_{L^{1}} \|FG\|_{L^1}~. 
 \end{align}
 This translates to 
 
   \begin{align}\notag
  \int_0^1 (s^{-\alpha} \Omega_r(s, f))^2  s^{r+\alpha} ds &  \leq  \left(\int_0^1   s^{2r+\alpha} ds\right)^2   \left(\int_0^1  s^{-\alpha} \Omega_r(s, f)  s^{2r+\alpha} ds\right)^2\\\notag
  & = c  \left(\int_0^1  s^{-\alpha} \Omega_r(s, f)  s^{2r+\alpha} ds\right)^2
 \end{align}
 for $c=c(\alpha, r)= (2r+\alpha+1)^{-2}$. By interfering this in (\ref{apply}) we get
 
   \begin{align}\notag 
\| \hat\psi_j(\Delta)f \|^2 
&\preceq  2^{-j(\alpha+M-2r)/2}   \|s\mapsto s^{-\alpha} \Omega_r(s, f)\|_{L^1([0,1],dm(s))}^2 , 
\end{align} 
or  equivalently, 
   \begin{align}\label{aster}
\| \hat\psi_j(\Delta)f \| 
 \preceq 2^{-j(\alpha+M-2r)/4}   \|s\mapsto  s^{-\alpha} \Omega_r(s, f)\|_{L^1([0,1],dm(s))}~. 
\end{align} 
  Therefore for $1<q<\infty$ and $q'=\cfrac{q}{q-1}$ 
  
   \begin{align}\notag 
\| \hat\psi_j(\Delta)f \| 
 &\leq  c ~ 2^{-j(\alpha+M-2r)/4}   \| s^{-\alpha} \Omega_r(s, f)\|_{L^1([0,1],dm(s))} \\\notag
 &\leq  c
 2^{-j(\alpha+M-2r)/4}   \| s^{-\alpha} \Omega_r(s, f)\|_{L^q([0,1],dm(s))}    \|  1 \|_{L^{q'}([0,1],dm(s))} \\\label{star}
 &\leq  c'  2^{-j(\alpha+M-2r)/4}   \| s^{-\alpha} \Omega_r(s, f)\|_{L^q([0,1],dm(s))} 
\end{align} 
 
 with $c'=c'(r,\alpha)= (r+\alpha+1)^{-2+1/q'}$.  By $0<s<1$, (\ref{star}) leads to 
   \begin{align}\notag 
\| \hat\psi_j(\Delta)f \| 
 &\leq  c'  2^{-j(\alpha+M-2r)/4}   \| s^{-\alpha} \Omega_r(s, f)\|_{L^q([0,1],ds/s)} 
\end{align} 
and hence 

\begin{align}\notag 
\| \hat\psi_j(\Delta)f \|^q \preceq
 c 2^{-jq(\alpha+M-2r)/4} 
\left( \int_0^{1} (s^{-\alpha} \Omega_r(s, f))^q ~  ds/s \right)~.
 \end{align}
 Therefore from above, for   $M> 3\alpha+2r$   we have

\begin{align}  \label{summation}
\sum_{j\geq 0} 2^{j\alpha q}\| \hat\psi_j(\Delta)f \|^q \leq 
\left( \sum_{j\geq 0} 2^{-j q(M-3\alpha-2r)/4}\right) \left( \int_0^{1} (s^{-\alpha} \Omega_r(s, f))^q ~ ds/s\right)= c   \left( \int_0^{1} (s^{-\alpha} \Omega_r(s, f))^q ~ ds/s\right).  
 \end{align}
  Or equivalently,  
\begin{align} 
\left(\sum_{j\geq 0} 2^{j\alpha q}\| \hat\psi_j(\Delta)f \|^q\right)^{1/q} \preceq
 \left( \int_0^{1} (s^{-\alpha} \Omega_r(s, f))^q ~ s^rds/s\right)^{1/q} 
 \end{align} 
 This completes the proof of Part II for $1< q<\infty$. Proof for $q=1$ can be obtained from the preceding calculations.  
 To prove the inequality (\ref{converse of first theorem}) for case $q=\infty$,  recall that in (\ref{aster})
   \begin{align}\notag 
\| \hat\psi_j(\Delta)f \| & \preceq 2^{-j(\alpha+M-2r)/4}
 \int_0^{1}  s^{-\alpha}\Omega_r(s, f)~  s^{2r+\alpha} ~ds\\\notag
 &\preceq  2^{-j(\alpha+M-2r)/4} ~ \sup_{0<s<1}( s^{-\alpha}\Omega_r(s, f))
 \end{align}
Therefore  for any $M> 3\alpha +2r$  
\begin{align}\notag 
\sup_{j\geq 0}  2^{j\alpha}\| \hat\psi_j(\Delta)f \|  ~ \preceq  ~
 \sup_{j\geq 0} (2^{-j(M-3\alpha -2r)/4} )~  \sup_{0<s<1}( s^{-\alpha}\Omega_r(s, f))~ \preceq~
   \sup_{0<s<1}( s^{-\alpha}\Omega_r(s, f)) ~,
\end{align}
and this completes the proof for $q=\infty$. 
\end{proof}

\section{   The Heisenberg group}\label{The Heisenberg group}

Let  
$\HH_n$ denote the  $(2n+1)$-dimensional  Heisenbeg group  $\HH_n$ identified by $\CC^n \times \RR$.  The multiplication is given by 
\begin{align}\notag
(z,t)\cdot(z',t')= \left(z+z', t+t-\frac{1}{2}Im\langle(z,z')\rangle\right)~\quad (z,t), (z',t')\in \HH_n~, 
\end{align}
with the identity $e=({\bf 0}, 0)$ and $(z,t)^{-1}=(-z,-t)$. \\
 
 Any $a>0$ defines an
automorphism of $\HH$ defined by 
\begin{align}\label{oneparameterauto}
a(z,t)= (az,a^2t)\hspace{.5in}\forall ~(z,t)\in\HH_n
\end{align}
where $az=(az_1,\cdots,az_n)$. 
 We fix the Haar measure $d\nu$ on the Heisenberg group which is the Lebesgue   measure $dzdt$ on $\CC^n\times \RR$.   
 For each $a>0$, the unitary dilation operator  $D_a$  on $L^2(\HH_n)$ is 
given by
\begin{align}
\delta_af(z,t) = a^{-(n+1)}f(a^{-1}z,a^{-2}t)\quad \forall f\in L^2(\HH_n),
\end{align}
 and for any $\omega\in \HH$, the left translation operator $T_\omega$ is   given by
\begin{align}\notag
l_\omega f(\upsilon)=  f( \omega^{-1}\upsilon)\quad\forall \upsilon\in \HH.
\end{align}

Define $\tilde f(z,t)= \bar f(-z,-t)$. We say $f$ is self-adjoint if $\tilde f= f$. It is easy to show that $\widetilde{\delta_a f}= \delta_a \tilde f $. 
 The \textit{quasiregular representation } $\pi$ of
the semidirect product   $G:=\HH_n\rtimes (0,\infty)$  acts on $L^2(\HH_n)$  by the dilation and translation operators, as follows. For any $f\in L^2(\HH)$ and $(\omega,a)\in G$  and $\upsilon\in\HH_n$ 
 \begin{align}\notag 
 (\pi(\omega,a)f)(\upsilon):=l_\omega \delta_af(\upsilon)= a^{-(n+1)}  f(a^{-1}(\omega^{-1}\upsilon)) .
 \end{align}

For $a>0$  define $f_a(z,t)= a^{-2(n+1)}f(a^{-1}z, a^{-2}t)$. Thus $\delta_af= a^{n+1}f_a$ and  $\tilde f_a= \widetilde{f_a}$.

 We let $\mathcal S(\HH_n)$ denote the space of Schwartz functions on $\HH_n$. By definition $\mathcal S(\HH_n)= \mathcal S(\RR^{2n+1})$.  Provided that the integral exists, 
 for any functions $f$ and $g$ on $\HH_n$, the convolution of $f$ and $g$ is defined by 
 \begin{align}\notag
 f\ast g(\omega)= \int_{\HH_n} f(\nu) g(\nu^{-1}\omega)~ d\nu .
 \end{align}

We
 fix the basis  $\frac{\partial}{\partial t}$ and $Z_j, \bar Z_j$, $j=1, \cdots, n$  for the  
 Lie algebra of $\HH_n$ where 
\begin{align}\notag
Z_j= 2 \frac{\partial}{\partial \bar z_j} + i \frac{z_j}{2}\frac{\partial}{\partial t},  \quad \bar Z_j= 2 \frac{\partial}{\partial z_j} - i \frac{\bar z_j}{2}\frac{\partial}{\partial t}
\end{align}
and $\frac{\partial}{\partial \bar z_j} $ and $\frac{\partial}{\partial   z_j} $ are  the standard derivations on $\CC$ and $\frac{\partial}{\partial t}$  is the derivation operator in direction $\RR$. 
These operators generate the algebra of left-invariant differential operators on $\HH_n$. 
Associated to this basis, the Heisenberg sub-Laplacian is defined by 
\begin{align}\notag
\Delta:= -\frac{1}{2}\sum_j (Z_j \bar Z_j + \bar Z_j Z_j). 
\end{align}
   $\Delta$ is self-adjoint and positive definite.  \\ 
 
 For general introduction to the Heisenberg group and its representations we refer to the papers of Geller \cite{Geller80,Geller84}.

  \subsection{  Gelfand pairs  associated to the Heisenberg group}\label{Gelfand pair}
  Let 
  $U(n)$ denote  the group of $n\times n$ unitary matrices. This group is   a maximal compact and connected  Lie subgroup of the automorphisms  group $\mathcal Aut(\HH_n)$  and 
    \begin{align}\notag
  \sigma(z,t)= (\sigma z,t)\quad  \forall ~ \sigma\in U(n),~~ (z,t)\in \HH_n ,
  \end{align}
  for any $z\in \CC^n$ and $t\in \RR$. These automorphisms are called rotations and are usually denoted by $R_\sigma$, instead. 
  All compact connected subgroups of $\mathcal Aut(\HH_n)$ can be obtained by conjugating $U(n)$ by an automorphism. \\

 Suppose    $K\subseteq U(n)$ is a Lie compact subgroup acting on $\HH_n$.   A function $f$ on $\HH_n$ is called {\it $K$-invariant}  if for any $k\in K$ and $\omega\in \HH_n$,~ $f(k\omega)= f(\omega)$. If we let $L^p_K(\HH_n)$ denote the $K$-invariant subspace of $L^p(\HH_n)$, then  for $K=\{I\}$ we have  $L^p_K(\HH_n)=L^p(\HH_n)$ where $I$ is the identity operator. For $K=U(n)$, the space $L^p_K(\HH_n)$ contains all rotation invariant elements in $L^p(\HH_n)$.  \\
 
For a subgroup $K$, the pair $(K, \HH_n)$ is called {\it Gelfand pair} associated to the Heisenberg group, or simply Gelfand pair, if  the space of  measurable $K$-invariant and integrable functions $L_K^1(\HH_n)$ is a commutative algebra with respect to the convolution operator, i.e., $f\ast g= g\ast f$.    It was known that $L^1_K(\HH_n)$ is a commutative algebra for $K=U(n)$, and thus $(U(n), \HH_n)$ is a Gelfand pair \cite{BJR92}.

   \subsection{$K$-spherical   transform}\label{spherical Fourier transform} 
      A smooth $K$-invariant function $\phi: \HH_n\rightarrow \CC$ is called   {\em $K$-spherical} associated to the Gelfand pair  $(K, \HH_n)$  if $\phi(e)=1$ and $\phi$ is joint eigenfunction for  all differential operators   on $\HH_n$ which  are invariant under the action of $K$ and $\HH_n$. 
  Equivalently, the $K$-spherical functions are homomorphisms of the commutative algebra $L_K^1(\HH_n)$.  The general theory of  $K$-spherical functions for Gelfand pairs $(K, \HH_n)$ was studied by Benson et al at \cite{BJRW98,BJR92} and \cite{TangaBook}.  \\

 The set of   $K$-spherical functions    associated to the pair $(K,\HH_n)$   is identified by  the space $\RR^*\times \NN^n$ and can be explicitly  computed  for concrete examples of $K$ (\cite{BJR92}).  The space $\RR^*\times \NN^n$ is called {\it Gelfand space}. 
 Let $\phi_{\lambda, {\bf m}}$ denote the  $K$-spherical function associated to $(\lambda, {\bf m})\in \RR^*\times \NN^n$.  
 Then 
for some polynomial $q_{\bf m}$  on $\CC^n$  with $|q_{\bf m}(z)|\leq |q_{\bf m}({\bf 0})|=1$, 
\begin{align}\notag
\phi_{\lambda, {\bf m}}(z,t)=e^{i\lambda t}e^{-|\lambda| |z|^2/4}q_{\bf m}(\sqrt{|\lambda|} z)~\quad \forall z\in \CC^n, ~ t\in \RR . 
\end{align}
 
The $K$-spherical functions  $\phi_{\lambda, {\bf m}}$ are eigenfunctions of the sub-Laplacian operator $\Delta $ with 
  eigenvalues  given by  
  \begin{align}\notag
  \Delta (\phi_{\lambda, {\bf m}})= |\lambda| (2|{\bf m}|+n) \phi_{\lambda, {\bf m}},
  \end{align}
  where $|{\bf m}|= |m_1|+\cdots +|m_n|$. 
  
{\bf Definition.}({\it $K$-spherical transform associated to $\Delta$})  The $K$-spherical transform $\mathcal F:=\mathcal F_K$ of a function $f\in L_K^1(\HH_n)$ at the character $(\lambda, {\bf m})\in \RR^*\times \NN^n$ is defined by
 
 \begin{align}\label{K-Fourier transform}
  \mathcal F( f)(\lambda, {\bf m}):= \int_{\HH_n} f(z,t) \phi_{\lambda, {\bf m}}(z,t) dzdt, 
 \end{align}
 
 where $dzdt$ is the Haar measure for  $\HH_n$ and the integral is well-defined.  
 
   \begin{lemma}\label{convolution}    For any $f$ and $g$ in $L^1_K\cap L^2_K(\HH_n)$
    \begin{align}\notag
        \mathcal F(f\ast g)(\lambda,{\bf m}) = \mathcal F(f)(\lambda,{\bf m}) \mathcal F(g)(\lambda,{\bf m})\quad a.e. ~~~ (\lambda,{\bf m}). 
    \end{align}
And, by density of $L^1_K\cap L^2_K(\HH_n)$ in $L^2_K(\HH_n)$ the assertion also holds for all $f$ and $g$ in $L^2_K(\HH_n)$. 
     \end{lemma}
     \begin{proof} This is straightforward from (\ref{K-Fourier transform}).
     \end{proof}
      
 The definition (\ref{K-Fourier transform}) implies that the spherical Fourier transform 
    $ \mathcal F( f)$ is bounded with $\| \mathcal F( f)\|_\infty \leq \|f\|_1$ and lies in $C_0(\RR^*\times \NN^n)$, the space of continuous functions with  fast decay at infinity. (We say a function $F:\NN^n\to \CC$ has fast decay at infinity if for any sequence $\{a_{\bf m}\}$  of complex numbers  with $|a_{\bf m}|\rightarrow \infty$ as $|{\bf m}|\rightarrow \infty$,
    $$ \lim_{|{\bf m}|\rightarrow \infty }|a_{\bf m} F({\bf m})| <\infty. )$$

 The  Godement's Plancherel theory for the Gelfand pairs (see \cite{God62}) guaranties the existence of a unique positive Borel measure $d\mu$ on $\RR^*\times \NN^n$ for which for all continuous functions  $f\in L_K^1(\HH_n) \cap L_K^2(\HH_n)$ 
  \begin{align}\label{isometry}
  \int_{\HH_n} | f(z,t)|^2 dzdt = \int_{\RR^*} \sum_{\NN^n} | \mathcal F( f)(\lambda,{\bf m})|^2 d\mu.
 \end{align}
     The Godement-Plancherel measure $d\mu$ on
 the Gelfand space
$\RR^*\times \NN^n$
is  given explicitly  by 
 \begin{align}\notag
 \int_{\RR^*} \sum_{\NN^n}  F(\lambda, {\bf m}) d\mu(\lambda, {\bf m}) = (2\pi)^{-(n+1)} \int_{\RR^*} \sum_{ \NN^n} w_{\bf m} F(\lambda, {\bf m})
|\lambda|^{n} d\lambda
 \end{align}
for some positive constant weights  $w_{\bf m}$ dependent  on degree of the polynomials $q_{\bf m}$ and $d\lambda$ is the Lebesgue measure on $\RR$. (For the proof of this see e.g.  \cite{BJRW98} and \cite{Yan}.)

The inversion formula for a continuous  function $f$ in $L_K^1(\HH_n) \cap L_K^2(\HH_n)$  with integrable $K$-spherical   transform  $\mathcal F( f)$ is given  by  
\begin{align}\label{inversion}
f(z,t)=  (2\pi)^{-(n+1)} \int_{\RR^*} \sum_{\in \NN^n} w_{\bf m}  \mathcal F( f)(\lambda,{\bf m}) \phi_{\lambda, {\bf m}}(z,t)~ |\lambda|^n d\lambda.
\end{align}
 The    transform $\mathcal F$ extends uniquely to a unitary operator between $L_K^2(\HH_n)$ and the Hilbert space
$L^2(\RR^*\times \NN^n, d\mu)$. 
  We let $\mathcal F$ denote this unitary operator.   Therefore the inversion formula (\ref{inversion}) also holds for all $f\in L^2_K(\HH_n)$ in the weak sense. \\

 Let $f\in L_K^2(\HH_n)$ and $k$ be a natural number. Then  $f$ lies in the domain of $\Delta^k$ if and only if the measurable map  $(\lambda, {\bf m})\mapsto |\lambda|^k (2|{\bf m}|+n)^k \mathcal F(f)(\lambda,{\bf m})$ is in $L^2(\RR^*\times \NN^n, d\mu)$. 
   This indicates that  the domain  $\mathcal D(\Delta^k)$ is transferred by   the unitary operator $\mathcal F$  onto the subspace of functions $f\in L^2(\HH_n)$ for which 
 
$$ \int_{\RR^*\times \NN^n}   |\lambda|^{2k} (2|{\bf m}|+n)^{2k} |\mathcal F( f)(\lambda,{\bf m})|^2 d\mu(\lambda,{\bf m})<\infty.$$

We let  $\Delta^k$ denote the closure of 
    $\Delta^k$ on $L_K^2(\HH_n)$.  The preceding also shows that  this operator   is uniquely equivalent to a multiplication operator $M_k$ acting on  $L^2(\RR^*\times \NN^n, d\mu)$ through $\mathcal F$ where  for any $F\in L^2(\RR^*\times \NN^n, d\mu)$ 
    $$M_k(F): ~ (\lambda, {\bf m})\mapsto  |\lambda|^k (2|{\bf m}|+n)^k F(\lambda, {\bf m}) .$$

From above we can conclude that if  $\beta\in L^\infty(0,\infty)$, then  for  any $f\in L_K^2(\HH_n)$
  
  \begin{align}\label{beta}
  \mathcal F(\beta(\Delta)f)(\lambda, {\bf m})= \beta(|\lambda|(2|{\bf m}|+n)) \mathcal F(f)(\lambda, {\bf m})\quad \text{a.e.} ~~~ (\lambda, {\bf m})
   \end{align}

         It is known by the spectral theory that for any  bounded  $\beta\in L^\infty(0,\infty)$   the operator $\beta(\Delta)$ is a bounded and kernel operator on $L^2$ with a kernel in $L^2$. The following constructive lemma shows how to obtain this kernel   using the $K$-spherical Fourier transform for the Gelfand pairs $(K, \HH_n)$. 
                  
   \begin{lemma}\label{kernel} Let  $\beta\in L^\infty(0,\infty)$. For $a>0$ define $\beta^a(\xi)= \beta(a\xi)$. Then the integral operator $\beta^t(\Delta)$ is a convolution operator and 
     there  exists $B\in L_K^2(\HH_n)$ such that for any $f\in L^2_K(\HH_n)$ 
$$\beta^a(\Delta) f = f\ast B_{\sqrt{a}}  $$
where $B_{ \sqrt{a}}(\omega)= a^{-(n+1)}B(a^{-1/2}\omega) $. If $\beta$ is real-valued, then  $\tilde B=B$. 
 
\end{lemma} 
     \begin{proof}
    Without loss of generality we prove that $\beta^{a^2}(\Delta)f= f\ast B_a$. 
    Let 
 $\alpha: \RR^*\times \NN^n\mapsto (0,\infty)$ be the measurable map defined  by 
 $\alpha(\lambda, {\bf m})= |\lambda| (2|{\bf m}|+n)$.   Since  $\beta$ is   bounded, the map  
  $\beta\circ \alpha: ~(\lambda, {\bf m})\mapsto  \beta(|\lambda|(2|{\bf m}|+n))$  is bounded, measurable,  and due to the finiteness of the measure $d\mu$ it  lies 
    in $L^2(\RR^*\times \NN^n, d\mu)$.  Let $B\in L^2_K(\HH_n)$   denote the spherical Fourier inverse of  $\beta\circ \alpha$.        Therefore   for a.e.   $(\lambda, {\bf m})$
    $$\mathcal F(B)(\lambda, {\bf m})=  \beta(|\lambda|(2|{\bf m}|+n)). $$  
    
    By the definition of dilation operator  and the preceding result, for any $a>0$ we have  
   \begin{align}\label{star}
   \mathcal F(B_{a})(\lambda, {\bf m})=  \beta(a^{2}|\lambda|(2|{\bf m}|+n))=  \beta^{a^{2}}(|\lambda|(2|{\bf m}|+n));
   \end{align}
   thus             $$\mathcal F^{-1}(\beta^{a^{2}} \circ \alpha)= B_a.   $$ 
   
    To complete the proof of the lemma, we need to show  that the operator $\beta^{a^{2}}(\Delta)$ is a convolution operator and for any $f\in L_K^2(\HH_n)$
      \begin{align}\label{convolution-op}
      \beta^{a^{2}}(\Delta)f= f\ast B_a 
      \end{align}
      in $L^2$-norm.  If we replace $\beta$ by $\beta^{a^2}$ in (\ref{beta}), for almost every  $(\lambda, {\bf m})$ we get 
     $$ \mathcal F(\beta^{a^{2}}(\Delta)f)(\lambda, {\bf m})= \beta^{a^{2}}(|\lambda|(2|{\bf m}|+n))  \mathcal F(f)(\lambda, {\bf m}).$$
     
    And, by (\ref{star}) 
          
      $$\mathcal F(B_a)(\lambda, {\bf m}) =\beta^{a^{2}}(|\lambda|(2|{\bf m}|+n)).$$ 
      
     Using this and Lemma \ref{convolution} we arrive at 
       
      \begin{align}\notag
      \mathcal F(\beta^{a^{2}}(\Delta)f)(\lambda, {\bf m})& = \mathcal F(B_a)(\lambda, {\bf m})  \mathcal F(f)(\lambda, {\bf m}) 
       =  \mathcal F(f\ast  B_a)(\lambda, {\bf m}) \quad a.e. ~~ (\lambda, {\bf m})~.
      \end{align}
 
 The inverse of spherical Fourier transform concludes that 
  $ \beta^{a^{2}}(\Delta)f = f\ast B_a$ in $L^2$ norm. The fact that $\widetilde B= B$ for the  real-valued function $\beta$ is an  immediate application of the spectral theory. 
     
      \end{proof}

     \begin{theorem}\label{Hulanicki}{(\cite{Hulanicki})} Let $\beta\in \mathcal S(\RR^+)$. Then $B$, the kernel of operator $\beta(\Delta)$, is in $\mathcal S(\HH_n)$ and 
      \begin{align}
      \beta(\Delta) f= f\ast B\quad \forall ~ f\in L^2 .
      \end{align}
             \end{theorem}
             
       \noindent
       {\bf Notation.}  In the sequel  we shall call $B$ the distribution kernel for $\beta$.   
       Note that in our situation the kernel $B$ is the inverse of $K$-spherical  transform of the map $\beta\circ \alpha$  (see above for the definition of $\alpha$). 
 And, from now on, if $\hat\psi\in L^\infty(0,\infty)$, then $\Psi\in L_K^2(\HH_n)$ denotes the distribution kernel of the operator $\hat\psi(\Delta)$. 

\subsection{Wavelets  for $L^2_K(\HH_n)$}\label{continuous wavelets}

The existence of wavelets in $L^2$ was proved by Liu-Peng \cite{Liu-Peng} for the Heisenberg group,  by F\"uhr \cite{Fuehrbook} (Corollary 5.28) for general homogeneous groups, and by Currey \cite{currey07} for nilpotent Lie groups. The construction of Shannon wavelets using multiresolution analysis method  was presented for the Heisenberg group in \cite{m1}. 

In contrast to those works, this article dose not use any representation theory. 
  The first wavelet systems on stratified Lie groups possessing a lattice were constructed by Lemari\'e \cite{lemarie}, by suitably adapting concepts from spline theory.
  More recent construction of both continuous and discrete wavelet systems were based on the spectral theory of the sub-Laplacian for stratified Lie groups  \cite{gm1}. Their wavelets are Schwartz and enjoy all vanishing moments, or compactly supported with arbitrary many vanishing moments. 
 Using the spectral theory methods, construction of smooth wavelets with more significant properties  were introduced  on compact and smooth manifolds    in \cite{gm2} and \cite{gm3}.  
    In this article, the construction of  continues and bandlimited  wavelets   and  the classification of Besov spaces    
 in terms of these wavelets are  based on the spectral theory techniques. The definition of a wavelet on $\HH_n$ is given as following. 
 
 We  say $\varphi\in L_K^2(\HH_n)$  is a continuous wavelet if for any $f\in L^2_K(\HH_n)$  the  isometry
\begin{align}\label{wavelet-isometry}
\|f\|^2 = c_\varphi \int_{\HH_n}\int_0^\infty |\langle f, l_w \delta_a\varphi  \rangle|^2 d\mu(a,w)
\end{align}
holds  for some   constant  $c=c_\varphi>0$, 
where 
 $d\mu(a,w)= a^{-(2n+3)}dadw$ is the left Haar-measure for the product group $\HH_n\times (0, \infty)$.  Associated to a   wavelet $\varphi$ and  a function  $f\in L^2_K$ we  define the coefficient map $W_{f,\varphi}:  \HH_n\times (0, \infty) \ni(w,a)\mapsto  \langle f, l_w \delta_a\varphi  \rangle$ and  call  
 $\langle f, l_w \delta_a\varphi  \rangle$ {\it  the continuous wavelet coefficient  } of  $f$ at   position $w$ and scale $a$.  The isometry (\ref{wavelet-isometry})  implies   the following inversion  formula 

\begin{align}\label{wavelet-reconstruction-formula}
f = c_\varphi^{\frac{1}{2}} \int_{\HH_n}\int_0^\infty \langle f, l_w \delta_a \varphi \rangle ~l_w \delta_a\varphi  ~d\mu(a,w) ,
\end{align}
where the integral is understood in the weak sense. (\ref{wavelet-reconstruction-formula}) can simply be interpreted   as reconstruction of $f$ through its wavelet coefficients $ \langle f, l_w \delta_a \varphi \rangle$  associated to the  wavelet $\varphi$.

  We say a measurable function  $ \nu: [0,\infty)\to \RR$   is {\it admissible} if   $\int_0^\infty |\nu(t)|^2  ~  dt/t<\infty $.  
    
\begin{theorem}\label{equiv-conditions}  Let $\alpha$ be  the measurable map  defined in Lemma \ref{kernel}. Let $\nu: [0,\infty)\to \RR$ be a measurable function and   $\nu\circ \alpha$ is in $L^2(\RR^*\times \NN^n)$.   Put $\varphi:=\mathcal F^{-1}(\nu\circ \alpha)$.  Then  $\varphi$ is a continuous wavelet 
 if and only if   $\nu$ is admissible. 
In this case $c_{\varphi}=\frac{1}{2} \int_0^\infty |\nu(t)|^2  ~  dt/t$. 
 \end{theorem} 
 
\begin{proof}       Put $Q:= 2n+2$. Then for any  $f\in L_K^2(\HH_n)$
 \begin{align}\notag
\int_0^\infty \int_{\HH_n} | \langle f, l_w\delta_a\varphi\rangle |^2 a^{-(Q+1)} dadw &=\int_0^\infty \int_{\HH_n} |  f\ast  \delta_a\varphi(w) |^2 a^{-(Q+1)} dadw \\\label{integral-norm}
&= \int_0^\infty \| f\ast  \delta_a\varphi\|^2 a^{-(Q+1)}  da\\\notag
\end{align}

By the Plancherel theorem and Lemma \ref{convolution} 
\begin{align}\notag
 (\ref{integral-norm}) & 
 = \int_0^\infty \|\mathcal F( f\ast  \delta_a\varphi)\|^2 a^{-(Q+1)} da\\\notag
  &=\int_0^\infty \int_{\RR^*}\sum_{\bf m} |\mathcal F( f)(\lambda, {\bf m})|^2 |\mathcal F( \delta_a\varphi)(\lambda, {\bf m})|^2 d\mu(\lambda, {\bf m})  a^{-(Q+1)} da \\\notag
    &=\int_0^\infty \int_{\RR^*}\sum_{\bf m} |\mathcal F( f)(\lambda, {\bf m})|^2 |(v^{a^{2}}\circ \alpha)(\lambda, {\bf m})|^2 d\mu(\lambda, {\bf m})  a^{-1} da \\\notag
    &= \frac{1}{2}  \int_0^\infty \int_{\RR^*}\sum_{\bf m} |\mathcal F( f)(\lambda, {\bf m})|^2 |(v^{t}\circ \alpha)(\lambda, {\bf m})|^2 d\mu(\lambda, {\bf m})   dt/t \\\notag
    &= \frac{1}{2}   \int_{\RR^*}\sum_{\bf m} |\mathcal F( f)(\lambda, {\bf m})|^2 \left(\int_0^\infty |(v^{t}\circ \alpha)(\lambda, {\bf m})|^2~  dt/t\right) d\mu(\lambda, {\bf m}) . 
\end{align}
By dilation invariant property of the measure $dt/t$, the inner integral is 
$$ \int_0^\infty |(v^{t}\circ \alpha)(\lambda, {\bf m})|^2~  dt/t = \int_0^\infty |\nu(t)|^2~ dt/t .$$
Interfering this in the previous calculations, we can immediately conclude that $\varphi$ is a wavelet if and only if  $\int_0^\infty |\nu(t)|^2~ dt/t <\infty$. 
 
\end{proof}
 
\begin{corollary}\label{coro} Let  $\nu \in L^\infty(0,\infty)$  and $\nu\circ \alpha\in L^2(\RR^*\times \NN^n)$. Then $\varphi=\mathcal F^{-1}(\nu\circ\alpha)$ is  a wavelet if one of the followings holds:\\
(a)  $\nu(0)=0$ \\
(b)  $\nu$ is with support away from zero. \\
  (c) $\nu$ is given by $\nu(\xi)=  \xi^k \nu_0(\xi)$ where $0\neq \nu_0 \in  \mathcal S(\RR^+)$.
  \end{corollary}

\subsection{Coefficient characterization of Besov norms}\label{Besov for the Gelfand-pair}

Let $\hat \psi$ and $\hat\psi_j$ be the same bounded and real-valued  measurable functions on $\RR^+$    introduced in  the proof of Lemma \ref{mainLemma1} in Section \ref{proof of main theorem}.
   If we let $\Psi$ and $\Psi_j$ denote the (distribution) kernel of  the operators  $\hat\psi(\Delta)$ and $\hat\psi_j(\Delta)$, respectively, (see the definition followed by Lemma \ref{kernel})  then we have 

\begin{lemma}\label{some-properties-of-psiJ} The following hold: \\
 (a) $\Psi$ is a wavelet\\
(b)  $\Psi\in \mathcal D(\Delta^k)$ for any natural number $k$, and\\
(c)  $\Psi_j= 2^{-j(n+1)} \delta_{2^{-j}}\Psi$. 
\end{lemma}

\begin{proof} (a) and (b) are trivial  by Corollary \ref{coro} and the bandlimitedness of $\Psi$.  The proof of (c) is deduced from the proof of Lemma \ref{kernel}   and the spectral theory.   
\end{proof}

Let $Y=L^{2,q}_{\alpha}(\HH_n\times \ZZ^+), ~\alpha>0, ~1\leq q <\infty,$ denote the space of measurable functions $F$ on $\HH_n\times\ZZ^+$ for which 
 
  \begin{align}\notag
  \|F\|_{Y}:=
 \left(\sum_{j\geq0}2^{-j(q(n+1)-\alpha)}   \left( \int_{\HH_n} |F(w,j)|^2 dw \right)^{q/2}
\right)^{1/q} <\infty ,
\end{align}
and with standard definition for $ q=\infty$. 
$\|\cdot \|_{Y}$ defines a complete norm for $Y$.   For any wavelet $\Psi$ and any function $f$ we shall call the map $\HH_n\times \NN^n \ni (w,j) \mapsto \langle f, T_wD_{2^{-j}} \Psi\rangle   $ the wavelet coefficient map and denote it by $W_{f,\Psi}$. 

 \begin{theorem}\label{Besov-interms-of-wavelet-coefficients}[Main Theorem]  There exists a bandlimited wavelet $\Psi\in  L_K^2(\HH_n)$   such that  
  $f\in L_K^2(\HH_n)$ is in Besov space $B_{2,q}^\alpha$ if any only if 
 its  wavelet coefficient map  $W_{f,\Psi}$    belongs to the Banach space $Y= L^{2,q}_{\alpha}$. And, 
  \begin{align}\notag
\|f \|_{B_{2,q}^\alpha}  
\asymp  
  \| f\|+ \left(\sum_{j\geq0} 2^{-j(q(n+1)-\alpha)}  \|f\ast \delta_{2^{-j}}  \Psi^*\|^q
\right)^{1/q}
\end{align}
\end{theorem}

 \begin{proof}
 
 Let $\hat\phi, \hat \psi, \hat\psi_j$ be the same functions introduced in Section \ref{ProofOfLemmas} for which  the equation 
   (\ref{resolution of identity}) holds. Let $\Phi$  denote the kernel of  operator $\hat\phi(\Delta)$. Then  by Lemma \ref{kernel}  we have 
\begin{align}\label{ConvolutionKernel}
\hat\phi(\Delta)f= f\ast \Phi, \quad \text{and}~~ \hat\psi_j(\Delta) f = f\ast \Psi_j \quad \forall ~ f\in L^2_K(\HH_n). 
\end{align} 
 This translates the Besov norm to 
 \begin{align}\notag
 \|f \|_{B_{2,q}^\alpha}   
\asymp  \| f\ast \Phi\|+ \left(\sum_{j\geq0} 2^{j\alpha} \|f\ast \Psi^j\|^q
\right)^{1/q}  
\asymp   \| f\|+ \left(\sum_{j\geq0} 2^{j\alpha} \|f\ast \Psi^j\|^q
\right)^{1/q} .
\end{align}
 This is same as to say that 
 \begin{align}\notag
 \|f \|_{B_{2,q}^\alpha}  
\asymp   \| f\|+ \left(\sum_{j\geq0} 2^{j\alpha} \left(\int_{\HH_n} |\langle f, l_w {\Psi^j}\rangle |^2 dw\right)^{q/2}
\right)^{1/q}~.
\end{align} 

 By Lemma \ref{some-properties-of-psiJ} we have $\Psi^j=2^{-j(n+1)} \delta_{2^{-j}}\Psi$. Therefore

 \begin{align}\notag
\|f \|_{B_{2,q}^\alpha}  
\asymp & 
   \| f\|+ \left(\sum_{j\geq0} 2^{-j(q(n+1)-\alpha)} \left( \int_{\HH_n} |\langle f, l_w \delta_{2^{-j}}  \Psi\rangle|^2 dw \right)^{q/2}
\right)^{1/q}\\\notag
= &   \| f\|+ \left(\sum_{j\geq0} 2^{-j(q(n+1)-\alpha)}  \|f\ast \delta_{2^{-j}}  \Psi^*\|^q
\right)^{1/q}
\end{align}
and this completes the proof of  the theorem. 
 \end{proof}

  \begin{lemma}
For given   $f\in L^2_K(\HH_n)$ and $a>0$, 
   $$\mathcal F(\delta_a f)(\lambda, {\bf m})= a^{(n+1)} \mathcal F(f)(a^2\lambda, {\bf m}) \quad {\text for~~ a.e.} ~~  (\lambda, {\bf m}).$$ 
  \end{lemma}
  \begin{proof} 
  The proof is straightforward by the definition of dilation $\delta_a$ and the $K$-spherical Fourier transform. 
  \end{proof} 
 
 Based on above lemma, for any $a>0$  we define the dilation $A_a$ of function $F$  on the Gelfand space $\RR^*\times \NN^n$ by $$A_aF(\lambda, {\bf m})= a^{-\frac{(n+1)}{2}}F(a^{-1}\lambda, {\bf m}).$$  
 Therefore the characterization of $K$-spherical Besov norms in terms of Gelfand transform of a wavelet  follows. 
 
 \begin{corollary}\label{MainResult} Let $\alpha>0$ and $\Psi$ be the same as in Theorem \ref{Besov-interms-of-wavelet-coefficients}. Then for any  $f\in L_K^2(\HH_n)$

 \begin{align}\notag
\|f \|_{B_{2,q}^\alpha}  
\asymp  
   \| f\|+ \left(\sum_{j\geq0} \left( 2^{-j((n+1)-\alpha/q)}  \|\mathcal F(f) A_{2^{2j}}\mathcal F(\Psi)\|\right)^q
\right)^{1/q}
\end{align} 
 \end{corollary}

 
 {\bf Acknoledgment.} 
  The author is grateful to Isaac Pesenson  for his helpful comments.

 \makeatletter
\renewcommand{\@biblabel}[1]{\hfill#1.}\makeatother


\begin{thebibliography}{2}


\bibitem{Akh}
J. ~Akhiezer, {\em Theory of approximation}, Ungar, NY, 1956.



\bibitem{BL}
J. ~Bergh, J. ~Lofstrom, {\em Interpolation spaces},
Springer-Verlag, 1976.

\bibitem{BS}
M. ~Birman and M. ~Solomyak, {\em Spectral thory of selfadjoint
operators in Hilbert space}, D.Reidel Publishing Co., Dordrecht,
1987.

\bibitem{BB}
P. ~Butzer, H. ~Berens, {\em Semi-Groups of operators and
approximation}, Springer, Berlin, 1967.



 \bibitem{BJRW98}   
  C. Benson, J. Jenkins, G. Ratcliff, {\em The spherical transforms of a Schwartz function on the Heisenberg group}, J. Funct. Ana. {\bf 154}, 379-423 (1998).
 

\bibitem{BJR92} C. Benson, J. Jenkins, G. Ratcliff, {\em Bounded K-spherical functions on Heisenberg groups}, J. Funct. Ana, {\bf 105}, 409--443 (1992).

\bibitem{CGT}
 J.~Cheeger, M.~Gromov, M.~Taylor, {\em Finite propagation speed, kernel
estimates for functions of the Laplace operator and the geometry
of complete Riemannian manifolds}, J. Diff. Geom. {\bf 17} (1982),
15--53.


\bibitem{C1}
 J.G.~Christensen, {\em Sampling in reproducing kernel Banach spaces on lie groups}, http://arxiv. 
org/abs/1008.0627. 


\bibitem{CO1}
J.G.~Christensen, G. ~Olafsson, {\em  Examples of coorbit spaces for dual pairs},  Acta Appl. Math.  107  (2009),  no. 1-3, 25Ð48.


\bibitem{CO2}
J.G. ~Christensen, G. ~Olafsson, {\em Coorbit spaces for dual pairs},
Applied and Computational Harmonic Analysis, in press.

\bibitem{CMO1} J.G. ~Christensen, A. ~ Mayeli,  G. ~Olafsson, {\em Coorbit description and atomic decomposition of Besov spaces}, preprint. 

\bibitem{currey07} B. ~ Currey, {\it Admissibility for a class of quasiregular representations}, Canadian Journal of Mathematics, {\bf 59}, No. 5 (2007) 917-942

\bibitem{Fuehrbook} H.~ F\"uhr, {\it Abstract Harmonic Analysis of Continuous Wavelet Transforms}, Lecture Notes in Mathematics, 1863, Springer, 2005.


\bibitem{FM1} H.~ F\"uhr, A.~ Mayeli, {\it{Homogeneous Besov spaces on stratified Lie groups and their wavelet characterization}}, to appear in JFSA. 

\bibitem{Geller80} D. Geller, {\em Fourier analysis on the Heisenberg group, I, Schwartz spaces}, J. Funct. Anal. {\bf 36}, 205-254, (1980).

\bibitem{Geller84} D. ~ Geller, {\it  Spherical harmonic, the Weyl transform and the Fourier transform on the Heisenberg group}, Canad. J. Math. 36 (1984), 615-684. 

\bibitem{gm1} D.~ Geller, A.~ Mayeli,   {\it{Continuous wavelets and frames on stratified Lie groups I.,}} {\bf J. Fourier Anal. Appl.} 12 (5), 543- 579, (2006)

\bibitem{gm2} D.~ Geller, A.~ Mayeli,  {\it{Continuous wavelets on compact manifolds}};  {\bf Math. Z.}  262 (2009), no. 4, 895--927

\bibitem{gm3} D.~ Geller, A.~ Mayeli, {\it{Nearly tight frames and space-frequency analysis on compact manifolds}}; {\bf Math. Z.}   236 (2009), no. 2, 235--264

\bibitem{gm4} D.~ Geller, A.~ Mayeli, {\it{Besov spaces and frames on compact manifolds}};   {\bf Indiana Univ. Math. J.}, Vol. {\bf 58}, No. 5, (2009)
 
\bibitem{God62}
R. ~ Godement, {\em A theory of spherical functions I}, Trans. Amer. Math. Soc. {\bf 73} (1962), 496-556. 


\bibitem{KPes}
S.~Krein, I.~Pesenson, {\em Interpolation Spaces and Approximation
on Lie Groups}, The Voronezh State University, Voronezh, 1990.

\bibitem{KPS}
S. ~ Krein, Y. ~Petunin, E. ~Semenov, {\em  Interpolation of
linear operators}, Translations of Mathematical Monographs, 54.
AMS, Providence, R.I., 1982.

\bibitem{Hulanicki} A. ~ Hulanicki, {\it A functional calculus for Rockland operators on nilpotent Lie groups}, Stud. Math. {\bf 78} (1984), 253-266.

\bibitem{lemarie} Lemari\'e, P.G., {\it Base d'ondelettes sur les groupes de Lie stratifis}. Bull. Soc. Math. France
{\bf 117} (1989), 211Ð232.

\bibitem{Liu-Peng} H. P.~ Liu, L.Z.~  Peng, {\it Admissible wavelets associated with the Heisenberg group}. Pacific J Math, 1997, 180: 101--123

\bibitem{m1}{\it{Shannon multiresolution anaylsis on the Heisenberg group}};  {\bf J. Math. Anal. Appl.} 348 (2), 671-684,
    (2008)

\bibitem{N}
S. M.~Nikolskii, {\em Approximation of functions of several
variables and imbedding theorems}, Springer, Berlin, 1975.


\bibitem{Pes1}
I.~Pesenson, {\em Interpolation spaces on Lie groups}, (Russian)
Dokl. Akad. Nauk SSSR 246 (1979), no. 6, 1298--1303.

\bibitem{Pes2}
I.~Pesenson, {\em Nikolski- Besov spaces connected with
representations of Lie groups}, (Russian) Dokl. Akad. Nauk SSSR
273 (1983), no. 1, 45--49.

\bibitem{Pes3}
I.~Pesenson, {\em The Best Approximation in a Representation Space
of a Lie Group}, Dokl. Acad. Nauk USSR, v. 302, No 5, pp.
1055-1059, (1988) (Engl. Transl. in Soviet Math. Dokl., v.38, No
2, pp. 384-388, 1989.)




\bibitem{Pes4}
I.~Pesenson, {\em On the abstract theory of Nikolski-Besov spaces}, (Russian)  Izv. Vyssh. Uchebn. Zaved. Mat.  1988,  no. 6, 59--68;  translation in  Soviet Math. (Iz. VUZ)  32  (1988),  no. 6, 80Ð92 


\bibitem{Pes5}
I.~Pesenson, {\em  Approximations in the representation space of a Lie group}, (Russian)  Izv. Vyssh. Uchebn. Zaved. Mat.  1990,  no. 7, 43--50;  translation in  Soviet Math. (Iz. VUZ)  34  (1990),  no. 7, 49Ð57.



\bibitem{Pes6}
 I.~Pesenson, {\em  The Bernstein Inequality in the Space of
Representation of Lie group}, Dokl. Acad. Nauk USSR {\bf 313}
(1990), 86--90;
 English transl. in Soviet Math. Dokl. {\bf 42} (1991).

\bibitem{PesTransAMS} I.~ Pesenson, {\em A sampling theorem on homogeneous manifolds}, Trans. of AMS, Vol. 352(9), (2000), 4257-4270.

\bibitem{Pes7}
I.~Pesenson, {\em A Discrete Helgason-Fourier Transform for
Sobolev and Besov functions on noncompact symmetric spaces},
Contemp. Math, 464, AMS, (2008), 231-249.

\bibitem{Pes8}
I.~Pesenson, 
{\em Paley-Wiener approximations and multiscale approximations in Sobolev and Besov spaces on manifolds},  J. Geom. Anal.  19  (2009),  no. 2, 390Ð419.


\bibitem{Pes11}
I.~ Pesenson, M. Pesenson, {\em Approximation of Besov vectors by Paley-Wiener vectors in Hilbert spaces}, arXiv:1104.0959, 2011. 

\bibitem{TangaBook} S.~ Thangavelu, {\em Harmonic analysis on the Heisenberg group},  Brikh\"auser, 1998. 

\bibitem{Tit}
E.~Titchmarsh, {\em Theory of Fourier Integrals}, Oxford
University Press, 1948.




\bibitem{Tr3}
H. ~Triebel, {\em  Theory of function spaces II,}
  Monographs in Mathematics, 84. Birkh\"user Verlag, Basel, 1992.
 
 \bibitem{Yan} Z. Yan, {\em Special functions associated with multiplicity-free representations}, preprint. 

\end{thebibliography}
   \end{document}